\documentclass{article}

\usepackage{amsfonts, amsmath, amsthm, amssymb}
\usepackage[all]{xy}

\newtheorem*{pr}{Proposition}
\newtheorem*{defi}{Definition}

\theoremstyle{remark}
\newtheorem*{remark}{Remark}

\newcommand{\cp}{\mathbb{C}_p}

\title{Twisted de Rham cohomology, homological definition of the
integral and ``Physics over a ring"}
\author{A. Schwarz, I. Shapiro}
\begin{document}
\maketitle
%\section{}
%\subsection{}
{\bf Abstract.}

We
define the twisted de Rham cohomology and show how to use it
to define the notion of an integral of the form
 $\int g(x) e^{f(x)}dx$  over an arbitrary ring. We discuss also a definition of a family of integrals and some properties of the homological definition of integral.
We show how to use the twisted de Rham cohomology in order to
define the Frobenius map on the p-adic cohomology.  Finally, 
we consider two-dimensional topological quantum field theories
with general coefficients.

\section {Introduction}

Physicists usually work with real or complex numbers.  It seems, however,  that the consideration of other numbers (for example,
$p$-adic numbers) can also be useful. This idea is not new (let us mention, for example, numerous papers devoted to $p$-adic
strings). It was conjectured that $p$-adic numbers and/or adeles
are relevant for the description of the space at small distances; this conjecture remains in the domain of speculations.  However, it was shown  that $p$-adic methods can be  used  as a mathematical tool that permits us to obtain information  about  theories  over the complex numbers. For example, in \cite{inst} the $p$-adic
analogue of the B-model was used to analyze integrality of instanton numbers. The $p$-adic B-model was defined there in completely formal way, however, one can conjecture that it has a more physical  definition and that  such a definition can lead to a deeper understanding of standard topological sigma-models. One can make a stronger conjecture that many other physical theories
can be formulated in the $p$-adic setting or, more generally, in the
setting when the role of numbers is played  by elements of any field
or even of any ring and that such a formulation can be used to obtain information about standard physical theories. The present paper is a step in this direction.

Let us emphasize that our goal is to get new insights into the structure of standard physical theories from number-theoretic considerations; but one can hope that  it is also possible to
apply  the ideas from physics to number theory.   It was found recently that S-duality of gauge theories can be used to
understand the geometric Langlands program \cite{KW}.  It is natural to
think that the Langlands program in number theory can also be
analyzed by means of a corresponding version of gauge theory.

We stressed that we want to use number theory in conventional physics.
It is possible, however, that  all physical quantities are quantized (there exists elementary length, etc). Then it is natural to believe that the theories over integers have direct physical meaning.

To explain what we have in mind when speaking about ``physics over a
ring" we start with the following:

\begin{defi}
Physics is a part of mathematics devoted to the calculation
of integrals of the form $\int g(x) e^{f(x)}dx.$ Different branches
of  physics are distinguished by the range of the variable $x$
and by the names used for $f(x),\, g(x)$ and for the integral.
For example, in  classical statistical physics $x$ runs over a symplectic manifold, $f(x)$ is called the Hamiltonian function
and the integral has the meaning of a partition function or
of a correlation function.   In a $d$-dimensional quantum field theory
$x$ runs over the space of functions on a $d$-dimensional
manifold (the space of fields) and $f(x)$ is interpreted as an action
functional.
\end{defi}

Of course this is a joke, physics is not a part of mathematics.
However, it is true that the main mathematical problem of physics
is the calculation of integrals of the form $\int g(x) e^{f(x)}\,dx.$
If we work over an arbitrary ring $K$ the exponential function and
the notion  of the integral are not defined. We will show that
nevertheless one can give a suitable definition of an integral of the form $\int g(x) e^{f(x)}\,dx.$

Let us start with some simple remarks about integrals over $\mathbb{R}^n$ assuming that $g$ and  $f$ are formal power series in the variable $\lambda$ with coefficients belonging to the ring
of polynomials on $\mathbb{R}^n$ (in other words $f, g\in \mathbb{R}[x^1,...,x^n][[\lambda]]$). We note that this choice is different from $\mathbb{R}[[\lambda]][x^1,...,x^n]$ and it is more convenient for technical reasons. If $f$ can be represented as $f_0+\lambda V$ where $f_0$ is a negative quadratic form, then
the integral  $\int g(x) e^{f(x)}\,dx$
can be calculated in the framework of perturbation theory
with respect to the formal parameter $\lambda$. We will fix $f$ and consider the integral as a functional $I(g)$ taking values in $\mathbb{R}[[\lambda]].$ It is easy to derive from the relation
$$\int \partial _a (h(x) e^{f(x)})dx=0$$ that the functional $I(g)$
vanishes in the case when $g$ has the form
$$g=\partial _a h +(\partial _a f) h.$$
One can show that this statement is sufficient to calculate
$I(g)$ up to a constant factor. This is roughly equivalent to the
observation that integration by parts is sufficient in this case to
determine the integral as a power series with respect to $\lambda$.  Later  we will derive the uniqueness of $I(g)$ from some general considerations; however, one should  notice that
one can give an easy elementary proof by induction
with respect to degree of the polynomial $g$. 

 One can consider the more general
integral
 \begin{equation}
\label{ }
\int  e^{f(x)}\rho
\end{equation}
as a functional $I(\rho)$ with the argument a form $\rho$ on $\mathbb{R}^n$. We assume that  $\rho$ is a $k$-form with
coefficients in $\mathbb{R}[x^1,...,x^n][[\lambda]]$ and the integrand is a closed form. The integration  is performed over a $k$-dimensional subspace of $\mathbb{R}^n$. However this integral does not vanish  (recall that $f=f_0+\lambda V$) only in the case
$k=n$, when it is essentially $I(g)$.

Let  us now consider a formal expression  $I(\rho)=\int  e^{f(x)}\rho$
where $x\in K^n$ and $\rho$ is a form on $K^n$ for an arbitrary ring $K$. We will assume that  $f$ and the coefficients of the form $\rho$ belong to the ring $K[x^1, ...x^n][[\lambda]]$.  Moreover, we will suppose that $f=f_0+\lambda V$ where $f_0=\frac{1}{2}x^t Ax$  and $A$ is an invertible matrix with entries from $K$.
We will \emph{define} $I(\rho)$ as \emph{a}  $K[[\lambda]]$-linear functional taking values in $K[[\lambda]]$ and vanishing on
\begin{equation}
\label{ }
\rho= dh+(df)h
\end{equation}
for an arbitrary form $h$.  We will prove that this definition specifies
$I(\rho)$ up to a constant factor on all  forms satisfying
$d\rho +(dh)\rho=0$, in particular on all $n$-forms. This statement can be reformulated in homological terms by
considering the twisted differential $d_f \rho=dh+(df)\rho$ on the space of all
differential forms in $x^i$.

We  can normalize the functionals $I(g)$ and $I(\rho)$ by
requiring that $I(g)=1$ if $g=1$ (or equivalently, $I(\rho)=1 $ if
$\rho=dx^1...dx^n$).
The normalized functionals are defined uniquely  in the setting of perturbation theory if $f$ is a perturbation of a non-degenerate quadratic form.

Notice that in the case when $K$ is a field one can use the standard Feynman diagram techniques with the propagator $A^{-1}$
and internal vertices specified by $V$ to calculate $I(g)$. The function $g$ determines external vertices of the diagram. To prove this statement we notice that the sum of Feynman diagrams
obeys
\begin{equation}
\label{d}
I(gx^k)=I(A^{ka}\partial_ag)+I(A^{ka}\partial_a (\lambda V) g).
\end{equation}
This follows from the remark  that multiplying $g$ by $x^k$ we add one external vertex to the diagram. The  diagrams for the new set of external vertices can be obtained from old diagrams by adding a new edge connecting the new external vertex to an old
(external or internal) vertex;  the first summand in the RHS of
(\ref {d}) corresponds to an edge ending in an external vertex, the
second summand to an edge ending in an internal vertex.
From the other side  (\ref {d}) is equivalent to the defining relation
for the functional $I(g)$. Considering only Feynman diagrams without connected components and having only internal vertices
 we obtain the normalized functional $I(g)$.

The interpretation of the defining relation for the functional $I(g)$ in terms of diagrams suggests a generalization of this relation to the case of infinite-dimensional integrals.

The paper is organized as follows. First of all we
define the twisted de Rham cohomology and show how to use it
to define the notion of an integral over an arbitrary ring. We discuss also a definition of a family of integrals and some properties of the homological definition of integral.
Then we show how to use the twisted de Rham cohomology in order to
define the Frobenius map on the p-adic cohomology.  Finally, in the last section
we consider two-dimensional topological quantum field theories
with general coefficients.

Throughout the paper we implicitly use the assumption that our algebraic manifolds are in fact affine varieties.  This is not a crucial assumption and is used only to streamline the exposition.  Whenever necessary the machinery of hypercohomology can be used to generalize the statements and constructions presented to the case of general algebraic manifolds.

\section {Twisted de Rham cohomology and the homological definition of the integral}

Let  us consider a  polynomial function $f(x)$ on the space $\mathbb{C}^n$. We define the twisted de Rham differential
as $$d_f=d+df$$ where $d$ stands for de Rham differential
and $df$ denotes  the operator of multiplication by the
one-form denoted by the same symbol. The twisted de Rham
cohomology $\mathcal{H}_f$ is defined as the cohomology
of the differential $d_f$ acting on the space of polynomial
differential forms on $\mathbb{C}^n$.

Notice that the restriction of the coefficients to polynomials is important: if we allow
forms with arbitrary holomorphic coefficients, then the cohomology of
$d_f$ is essentially trivial because the new differential is equivalent to the standard de Rham differential, namely $$d_f=e^{-f}\circ d\circ e^f.$$

It is easy to construct a linear functional on $\mathcal{H}_f$ starting with  a singular  cycle $\Gamma$ in $\mathbb{C}^n$ having the
property that  the function $e^f$ tends to zero faster than
any power of $\|x\|$ when $x\in \Gamma$ tends to infinity.
Namely, such a functional can be defined
by the formula
 \begin{equation}
\label{I}
I([u])=\int _{\Gamma}{u e^f}
\end{equation}
where $u$ is a polynomial differential form obeying
$d_f u=0$ and $[u]$ stands for its class in $\mathcal{H}_f$.  The condition on the cycle $\Gamma$ ensures that the integral in (\ref{I}) makes sense, and the fact that the functional $I$ does not depend on the choice of the representative of the cohomology class $[u]$ follows as usual from the Stokes' theorem.

Moreover, one can show that every linear functional on
$\mathcal{H}_f$ is a linear combination of functionals of this kind. Thus $\mathcal{H}_f$ captures exactly the minimal amount of information that is required to compute the integral over any possible contour.  More precisely, if $X$ is a smooth algebraic variety  and $f$ is an algebraic function on $X$, then there is a non-degenerate pairing between the singular homology of the pair $(X(\mathbb{C}), f^{-1}(\{z\in\mathbb{C}|-\mathfrak{R}(z)>C\gg 0\}))$ (which we will denote by $\mathcal{H}^f$) and $\mathcal{H}_f$ where $X(\mathbb{C})$ is viewed as an analytic manifold and
$\mathcal{H}_f$ is defined by means of differential forms with
algebraic coefficients (see for example \cite{kz}).\footnote{One may need to take hypercohomology in the definition of $\mathcal{H}_f$ if $X$ is not an affine variety.} The singular homology with
integral coefficients specifies a lattice  in
$\mathcal{H}^f$; we say that the elements of this lattice are topologically integral.

If the function $f$ has only a finite number of critical points  one can prove that the  cohomology
$\mathcal{H}_f$ vanishes in all dimensions except $n$ and
that in dimension $n$ it is isomorphic to the quotient of the
polynomial ring $\mathbb{C}[x_1,..., x_n]$ by the
ideal generated by the derivatives of $f$ with respect to $x_1,...,x_n$ (to the Milnor ring, or in another terminology Jacobian ring). Under certain conditions one can prove that the dimension of $\mathcal{H}_f$  coincides with the dimension of the cohomology of the operator $df$ (Barannikov-Kontsevich theorem).\footnote{This statement  was proven in the case when $f$ is a regular projective function in \cite{SA}; another proof is given in \cite {OgV}.
It is conjectured, see \cite{SA}, that it is also true in the case when the intersection of $f^{-1}(0)$ with the set of critical points of $f$ is projective. One can characterize the dimension of $\mathcal{H}_f$ also as the total number of vanishing cycles for all the critical values of $f$.}

In the case of a finite number of critical points the cohomology of $df$ is concentrated in the dimension $n$ where
it coincides with the Milnor ring;
we obtain the description of $\mathcal{H}_f$ given above from the Barannikov-Kontsevich theorem.

More generally, we can also define $\mathcal{H}_f(K)$  as the cohomology of the differential $d_f=d+df$
in the case when $f\in K[x_1,...x_n]$, i.e., $f$ is a polynomial function on $K^n$ where $K$ is an arbitrary ring. Furthermore,
$f$ can be an algebraic function on a manifold over $K$, then
$\mathcal{H}_f$ should be defined as hypercohomology of
the (differential graded) sheaf of forms equipped with the differential $d_f$.

The above considerations prompt the following definition of the
integral of the form (\ref {I}) where $f$ is a polynomial function
on $K^n$ (or an algebraic function on a manifold over $K$) and
$\rho$ is a differential form on $K^n$ with polynomial coefficients
(or an algebraic differential form on a manifold) obeying
$d_f \rho=0$ (notice that an $n$-form always obeys this condition).  Namely,
we define an integral as a $K$-linear functional on $\mathcal{H}_f (K)$.

Below we restrict our attention to the case when
 $f$ is a polynomial function
on $K^n$.

The definition of $\mathcal{H}_f$ and of the  integral can also be
applied to the slightly more general case when  the coefficients of the
polynomial $f$ are not necessarily in $K$, but the coefficients
of the form $df$ belong to $K$.  This is due to the fact that $f$ itself appears nowhere in the definitions.

In the definition of the integral it is natural to require that $K$ is a ring without torsion elements  (i.e., it injects into $K\otimes_{\mathbb{Z}}\mathbb{Q}$) and to  neglect torsion elements in  $\mathcal{H}_f (K)$, i.e., to consider the integral as an element of  the quotient of $\mathcal{H}_f(K)$ with respect to its torsion subgroup; the quotient can be interpreted as the image of
$\mathcal{H}_f(K)$ in $\mathcal{H}_f(K\otimes_{\mathbb{Z}}\mathbb{Q})
\cong\mathcal{H}_f(K)\otimes_{\mathbb{Z}}\mathbb{Q}$.  In what
follows we use the  notation $\mathcal{H}'_f(K)$  for this quotient.

Notice that the definitions of $\mathcal{H}_f$ and, to a certain extent, of the  integral are functorial.  This means that, for example, a homomorphism of
rings $K\to K'$ maps a polynomial $f$  on $K^n$ to a polynomial $f'$ on $K'^n$ and $\mathcal{H}_f (K)$ to $\mathcal{H}_{f'}(K')$.  More precisely, we recall that $\mathcal{H}_f (K)$ is computed as the cohomology of a certain twisted de Rham complex, with free $K$-modules in each degree.  It is immediate that $\mathcal{H}_{f'}(K')$ is computed as the cohomology of the complex obtained from the one above, by tensoring it with $K'$ over $K$.  Thus we have a natural map of complexes and so of their cohomologies as well.  We note that in general $\mathcal{H}_{f'}(K')$ cannot be identified with $\mathcal{H}_f (K)\otimes_K K'$ and so a $K$-integral, which is a linear map from $\mathcal{H}_f (K)$ to $K$ cannot be extended to $\mathcal{H}_{f'}(K')$.  However this is possible in the case when $K'$ is flat over $K$ (this is the case if we take our $K$ to be $\mathbb{Z}$ and $K'$ to be our torsion free $K$ as above).  In this case it is true that $\mathcal{H}_{f'}(K')\cong\mathcal{H}_f (K)\otimes_K K'$.  It is also possible if $\mathcal{H}_f (K)$ is concentrated in degree $n$ since it is always true that $\mathcal{H}^n_{f'}(K')\cong\mathcal{H}^n_f (K)\otimes_K K'$.

In other words, the integrals for different rings are related. Moreover, if the polynomial $f$, or at least the form $df$ has
integer coefficients, the study of the integral can be reduced
to the case of the ring of integers   $ \mathbb{Z}$. This follows
from the remark that the ring of polynomial forms on $K^n$
can be realized as a tensor product of $K$ and the ring of polynomial forms
over $\mathbb{Z}$; this allows us to apply the universal coefficients theorem for the calculation of  $\mathcal{H}_f(K)$.

In particular, if $f$ is such a polynomial that $df$ has integer
coefficients  we can consider $\mathcal{H}_f (K)$ for an arbitrary
ring $K$; it follows from the universal coefficients theorem that for torsion free ring $K$
 \begin{equation}
\label{kk}
\mathcal{H}_{f} (K)=\mathcal{H}_f(\mathbb{Z})\otimes K
\end{equation}
and also

\begin{equation}
\label{k}
\mathcal{H}'_{f} (K)=\mathcal{H}'_f(\mathbb{Z})\otimes K.
\end{equation}

Notice that the group  $\mathcal{H}_f (\mathbb{Z})$ can be very poorly behaved.  More precisely, it can have not only an infinite number of generators, but also a lot of torsion.{\footnote {However,  in the case when  $f$ is a proper map and $p$ is sufficiently large one can prove that the group $\mathcal{H}_f(\mathbb{Z}_p)$  has a finite number of generators \cite{OgV}. }}

Note that $\mathcal{H}'_f(\mathbb{Z})$ on the other hand is a free abelian group of the rank equal to the dimension of
$\mathcal{H}_f(\mathbb{Q})$ and the embedding of $\mathcal{H}'_f(\mathbb{Z})$ into $\mathcal{H}_f(\mathbb{Q})$ specifies an integral structure therein. We say that an element of $\mathcal{H}_f(\mathbb{Q})$ or of $\mathcal{H}_f(\mathbb{C})$ is integral (or, more precisely, algebraically integral) if it belongs to the image of  $\mathcal{H}'_f(\mathbb{Z})$.  Notice that there exists no simple relationship between the topological and the algebraic integrality. Considering the pairing between the topologically integral elements of
$\mathcal{H}^f(\mathbb{C})$ and the algebraically
integral elements of $\mathcal{H}_f(\mathbb{C})$
we obtain in general transcendental numbers called exponential periods  (see \cite{kz}).

An important case of a quadratic $f$ is addressed below.  The following proposition simply means that the integral is determined up to a multiplicative constant in the setting of the quadratic exponential.

\begin{pr}Let $A$ be an invertible symmetric matrix with coefficients in $R$.
Let $f=\frac{1}{2}x^t Ax$ then $\mathcal{H}_f(R)$ (i.e. the cohomology of
$R[x_1,...,x_n][dx_i]$ with the differential $d+df$) is concentrated in degree $n$ and is isomorphic to $R$.\end{pr}

\begin{proof}
Let $A=(a_{ij})$ and $A^{-1}=(a^{ij})$.  Thus we are interested in
computing the cohomology of $R[x_i][dx_i]$ with the differential
$\sum_i (a^{ij}\partial_j+x_i)\otimes a_{ij}dx_j$ where we omit
$\sum$ when an index is repeated.  This cohomology is the same (as a
graded module) as that of $R[x_i][\xi_i]$ with the
differential $\sum_i (a^{ij}\partial_j+x_i)\otimes \xi_i$, where
$\xi_i=a_{ij}x_j$ since $A$ is invertible.  It is easy to verify
that $D_i=a^{ij}\partial_j+x_i$ form a regular sequence of commuting
operators on $R[x_i]$, and furthermore
$R[x_i]/(D_1,...,D_s)=R[x_{s+1},...,x_n]$. Thus $R[x_i][\xi_i]$ and
so $R[x_i][dx_i]$ has cohomology only in the top degree and it is
$R$.
\end{proof}

Let us now consider the setting of perturbation theory assuming that the unperturbed theory is given by a quadratic form with an invertible matrix.  The discussion that follows implies that also in this case the homological definition of the integral specifies
it up to a factor, and uniquely if we add the normalization
condition $I(dx_1 ... dx_n)=1$.

To construct the perturbation theory for $f=f_0+\lambda V$ it is convenient to work in terms of the ring $R_N(\lambda)$ by which we mean the quotient of the polynomial ring $R[\lambda]$ by the ideal generated by $\lambda^N$.  After allowing $N$ to go to infinity we can consider $\lambda$ as a parameter
of perturbation theory. Under certain conditions, in particular, in the case of the perturbation of a quadratic form $f_0={\frac{1}{2}x^t Ax}$, one has
that the cohomology $\mathcal{H}_{\frac{1}{2}x^t Ax+\lambda
V}$ is isomorphic to $\mathcal{H}_{\frac{1}{2}x^t Ax}$.

To prove this we observe that the complex $C_f$ that computes $\mathcal{H}_f(R_N(\lambda))$ is
filtered by the powers of the parameter $\lambda$, let us call this filtration $F$.  When the
associated spectral sequence degenerates, i.e. the cohomology of
the complex $C_f$ is isomorphic to the cohomology of the
associated graded complex $Gr^F C_f$;  then we get the desired isomorphism since $Gr^F C_f$ computes $\mathcal{H}_{f_0}(R_N(\lambda))$.

Let us return to the above proposition and replace $R$ by $R_N(\lambda)$. The matrix $A$ is still an invertible symmetric matrix with coefficients in $R$; thus it remains so in $R_N(\lambda)$ as well. We see that $\mathcal{H}_{\frac{1}{2}x^t Ax}(R_N(\lambda))$  is concentrated in degree $n$ and isomorphic to $R_N(\lambda)$. This leads to  the degeneration of the spectral sequence since $Gr^F C_f$ has cohomology (which is just $\mathcal{H}_{\frac{1}{2}x^t Ax}(R_N(\lambda))$) concentrated in one degree only. Thus $\mathcal{H}_{\frac{1}{2}x^t Ax+\lambda
V}(R_N(\lambda))$ is also concentrated in one degree; where it is isomorphic to $R_N(\lambda)$.

There is another way to identify $\mathcal{H}_f(R_N(\lambda))$ with $\mathcal{H}_{f_0}(R_N(\lambda))$ that we will not use.  Namely, if $R$ is a $\mathbb{Q}$-algebra, i.e. all the integers are invertible in $R$, then we always have a canonical identification $\mathcal{H}_f\cong\mathcal{H}_{f_0}$ via the exponential map.  Explicitly we have $\phi:\mathcal{H}_f\rightarrow\mathcal{H}_{f_0}$ with $\phi(\omega)=e^{f_1 \lambda+f_2\lambda^2+ ...}\omega$.  Unfortunately this method destroys any integrality information.

\begin{remark}
By letting $N$ go to infinity in the definition of $R_N(\lambda)$ we pass outside the considerations of polynomial differential forms with arbitrary coefficients.  This is due to the simple observation that $R[[\lambda]][x_i]$ which does fit into our framework, is not the same as  $R[x_i][[\lambda]]$ which is what we obtain after taking the limit of $R_N(\lambda)$'s.  This is the first example of the relaxation of the polynomial condition.  We will see another one later when we encounter overconvergent series.
\end{remark}

One can apply the above statements to  prove some integrality
results.  Namely the proposition above and the discussion that follows it shows that, when the polynomials $g$ and $h$ have integer coefficients, and the symmetric negative definite matrix  $A$ has integer entries and is invertible (over $\mathbb{Z}$), then the quotient $$\dfrac{\int_{\mathbb{R}^n}g(x)e^{\frac{1}{2}x^t Ax+\lambda h(x)}\,dx}{\int_{\mathbb{R}^n}e^{\frac{1}{2}x^t Ax+\lambda h(x)}\,dx}$$ is a power series in $\lambda$ with integer coefficients. (In fact only the one-form $dh$, and not $h$ itself, needs to have integer coefficients.) The expression at hand is a normalized  homological integral over $\mathbb{Z}$. As we mentioned in the introduction  the proof of uniqueness in the framework of perturbation theory can be given by induction with respect to degree of polynomial $g$. Analyzing this proof one can construct a version of perturbation theory for 
normalized integral dealing only with integers.
(In the standard Feynman perturbation expansion the integrality is not obvious.)

Notice that the appropriate apparatus for the study of the groups
$\mathcal{H}_f$ is the theory of D-modules \cite{B}, \cite{borel}. By definition a D-module
is  a sheaf of modules over the sheaf of rings of differential
operators.  In the case of interest, namely the linear situation that we focus on, the structure of a D-module is equivalent to the action of the Weyl algebra.  Recall that the Weyl algebra is generated over the constants by the symbols $x_i$ and $\partial_{i}$ subject to the relations $$\partial_{i} x_j-x_j \partial_{i}=\delta_{ij}.$$ The most obvious example of a module over this algebra is the ring of polynomials in the variables $x_i$; we will denote it by $\mathcal{O}$. An important $D$-module for us is a modification of this construction.  Namely, any polynomial  $f$ specifies a new D-module structure on $\mathcal{O}$ (usually denoted by $\mathcal{O}e^f$)  with the action of $x_i$ (viewed as elements of the Weyl algebra) unchanged, i.e., given by the operators of multiplication by the corresponding coordinates, while $\partial_i$ now acts as $\frac{\partial}{\partial x_i}+\frac{\partial f}{\partial x_i}\cdot$. The twisted cohomology coincides with the cohomology of this D-module.

The notion of a D-module is a generalization of the notion of a vector bundle with a flat connection. From a somewhat different point of view, a $D$-module encodes a system of linear differential equations. Thus while a certain function may not exist algebraically, if it is a solution of a
linear system of algebraic differential equations, one can consider
a $D$-module that is associated to this system.  In particular the function
$e^f$ is a solution of $\partial y-(\partial f) y=0$; the corresponding
D-module was described above as the D-module giving the twisted de Rham cohomology $\mathcal{H}_f$ as its usual de Rham
cohomology.

It is important to emphasize that the above definitions can be
modified in many ways. In particular, in the definition of
$\mathcal{H}_f$ one can consider the differential $d_f$ acting on a
space  of forms that is larger than the space of forms with
polynomial coefficients. The most important case is the case of
$K=\mathbb{C}_p$ (of the  field of complex p-adic numbers) when  it
is convenient to work with overconvergent series instead of
polynomials. One says that a series $\sum a_I x^I$ is
overconvergent if $$\text{ord}_p a_I\geq c|I|+d$$ with $c>0$, i.e. $\sum a_I
x^I$ converges on a neighborhood of the closed polydisc of radius
$1$ around $0\in\cp^n$.  Let us denote by $\mathcal{H}_f^\dagger$ this particular modification.  It is certainly not the case that one can always identify $\mathcal{H}_f$ with $\mathcal{H}_f^\dagger$, in general there is only an obvious map from one to the other that need not be surjective or injective.  In fact it seems that sufficiently general criteria for addressing this issue are not known. One can prove however that this replacement does
not change cohomology in certain important special cases, such as in the discussion surrounding the construction of the Frobenius map in the next section.

Recall that in the setting of perturbation theory and the quadratic exponential we have an essentially unique integral.  This is not true in general and reflects the freedom of choice of an integrating contour.  However, in certain cases we may use additional data to either ensure uniqueness or at least decrease the number of the available options.  For example, if we have a group $G$-action that preserves the function $f$, then this induces an action of $G$ also on $\mathcal{H}_f$.  It is then natural to require that the integral be invariant under this action.  This turns out to be of limited use since it is easy to see that the action of a Lie algebra on $\mathcal{H}_f$ is necessarily trivial (see equation (\ref{homotopy}) below).  Thus if $G$ is a connected Lie group then this does not cut down our choices.

A more interesting case comes up when one considers the integral in families.\footnote{We will return to this in the last section.}  More precisely, given a family of manifolds $X_\lambda$ equipped with functions $f_\lambda$ over a smooth parameter space $\Lambda$, we can consider the construction of $\mathcal{H}_{f_\lambda}$ on each $X_\lambda$. It is implicitly assumed that this arises from a smooth map of smooth spaces $p:X\rightarrow \Lambda$ and an $f$ on $X$, with $X_\lambda=p^{-1}(\lambda)$ and $f_\lambda=f|_{X_\lambda}$. If suitable conditions on the variation of $X_\lambda$ and $f_\lambda$ are imposed, then what one gets is a family of vector spaces of a fixed dimension that vary with $\lambda\in \Lambda$.  In other words we have a vector bundle over $\Lambda$; we will denote it by $\mathcal{H}_{f/\Lambda}$. An integral
depending on a parameter $\lambda\in \Lambda$ is then  a section of the dual bundle $\mathcal{H}_{f/\Lambda}^*$.

In fact this vector  bundle (and thus its dual) comes with a flat connection generalizing the Gauss-Manin connection.  This follows from a general fact that if one has a $D$-module on $X$, then by computing its fiberwise (along $p$) de Rham cohomologies we get a (graded) $D$-module  on the base $\Lambda$.  This does not require any extra assumptions on the variation of $X_\lambda$ and $f_\lambda$ and so the resulting $D$-module in general will not be a vector bundle with a connection.  We sketch a construction of this structure in our special case below.

Given a vector field $\xi$ on $\Lambda$ we must specify its action on $\mathcal{H}_{f/\Lambda}$.  We do this as follows.  Consider a lifting of $\xi$ to a vector field $\widetilde{\xi}$ on $X$.  Let $\xi$ act on $\mathcal{H}_{f/\Lambda}$ by \begin{equation}\label{action}L_{\widetilde{\xi}}+\widetilde{\xi}(f)\end{equation} where $L_{\widetilde{\xi}}$ denotes the Lie derivative with respect to $\widetilde{\xi}$ acting on the space of forms on $X$ (more precisely relative forms on $X$ over $\Lambda$). Observe that this is independent of the choice of the particular lifting of $\xi$ as follows from the formula
\begin{equation}\label{homotopy}
\{d_\Lambda+d_\Lambda f,\iota_\eta\}=L_{\eta}+\eta(f)
\end{equation}
where $\{,\}$ denotes the anti-commutator, $d_\Lambda$ the fiberwise de Rham differential\footnote{Recall that $d_\Lambda+d_\Lambda f$ is the differential in the complex that computes $\mathcal{H}_{f/\Lambda}$.} and $\eta$ any vertical vector field (i.e. a vector field tangent to the fibers of $p$).  Thus the action of a vertical vector field given by the equation (\ref{action}) is trivial on the cohomology. Note that precisely such an $\eta$ arises as the difference between any two choices of the lifting of $\xi$.  In the case when only the $f_\lambda$ vary and $X_\lambda$ remain constant, i.e., $X=Y\times\Lambda$ there is a natural lifting of vector fields from $\Lambda$ to $X$ that makes actual computations simpler.  The formula for the connection of course remains the same, but the Lie derivative can
be interpreted as a usual derivative with respect to the parameters.

It is natural to require that the $\lambda$-dependent integral is covariantly constant with respect to the generalized Gauss-Manin connection (in other words it specifies a flat section of the bundle $\mathcal{H}_{f/\Lambda}^*$). In some cases this assumption, together with the requirement that the section be single-valued and behave nicely at the boundary of the parameter space, determines the integral up to a constant factor. Let us give some details in the case when the coefficient ring is $\mathbb{C}$.

The topologically integral sections of $\mathcal{H}_{f/\Lambda}^*$\footnote{Recall that they correspond to singular cycles over $\mathbb{Z}$ in appropriate relative homology.} are covariantly constant (but in general multi-valued).  Using this remark one can
obtain differential equations for the periods (Picard-Fuchs equations) from the Gauss-Manin connection.
If the Picard-Fuchs equations have unique (up to a factor) single-valued solution we can say that the integral is also defined up to a factor.
For example this is true in the case when  the Gauss-Manin connection has maximally unipotent monodromy at the point $\lambda=0$.

One can prove some standard properties of the integral using the homological definition.  We will formulate these properties
as theorems about the groups $\mathcal{H}_f$.

1. Additivity.
The  property
$$ \int_{A\cup B}=\int _A+\int _B - \int _{A\cap B}$$
corresponds to the Mayer-Vietoris exact sequence.  

2. Change of variables takes the following form.  If
$\varphi:X\rightarrow Y$ and $f$ is a function on $Y$, then the
usual pullback via $\varphi^*$ of forms induces a map from
$\mathcal{H}_f$ to $\mathcal{H}_{f\circ\varphi}$.

3. Fubini theorem is replaced by a spectral sequence.  More
precisely, to compute $\mathcal{H}_f$ with $f$ a function on $X$ (let us denote it by $\mathcal{H}_f(X)$ to make explicit its dependence on $X$) we
compute the cohomology of a certain complex which in the case of a
decomposition of the space into a direct product, i.e. $X=X_1\times
X_2$, decomposes naturally into a double complex.  The associated
spectral sequence is thus the replacement for Fubini theorem.
Anything that leads to the degeneration of the spectral sequence is
beneficial, in particular the case when the cohomology is
concentrated in only one degree is especially similar to the
familiar Fubini theorem.  In the case when $f=f_1+f_2$ with $f_i$ a function on $X_i$ we have that $C_f(X)\cong C_{f_1}(X_1)\otimes C_{f_2}(X_2)$ and so $$\mathcal{H}_f(X)\cong\mathcal{H}_{f_1}(X_1)\otimes\mathcal{H}_{f_2}(X_2).$$

4. Fourier transform and the $\delta$ function. Let us consider a function $$f(t, x_1,...,
x_n)=itP(x_1, ...,x_n)$$ on $\mathbb{R}^{n+1}$. If in the integral
(\ref {I}) with this function the form $u$ does not depend on $t$ we can do an
integral over $t$; the $\delta$-function we obtain reduces the
integral to the integral over the hypersurface $P=0$. Therefore one
should expect that the cohomology $\mathcal{H}_f$ in this case is
isomorphic to the cohomology of the hypersurface. This statement was
proven in \cite{katz} (in different terminology); other proofs were given
in \cite{dworkdmod}, \cite{dworkdmod2}. We sketch a proof below.

\begin{proof}
The cohomology of the hypersurface $P=0$ is given by the complex
$\Omega /(P,dP)$ with the differential inherited from the usual
de Rham differential $d$ on the space $\Omega$
of differential forms on $\mathbb{R}^n$.  The claim is that this
is isomorphic (up to shift) to the cohomology of the space
$\Omega'$ of differential forms on $\mathbb{R}^{n+1}$ with the twisted differential $d+d(tP)$, where $t$ is the variable on the extra copy of $\mathbb{R}$.  The
intermediate step is the complex $\Omega[P^{-1}]/\Omega$
with the differential coming from $d$ extended to
$\Omega[P^{-1}]$ by the quotient rule.  The claim is
demonstrated by the following two maps, each of which induces an
isomorphism on cohomology.  The first is
$$\Omega/(P,dP)\rightarrow\Omega[P^{-1}]/\Omega$$ $$\omega\mapsto\tilde{\omega}
dP/P$$ and the second is
$$\Omega[P^{-1}]/\Omega\rightarrow\Omega'$$ $$\omega/P^{i+1}\mapsto(-1)^i\omega t^i/i!dt$$ thus the composition is simply $$\omega\mapsto\tilde{\omega}dPdt.$$
\end{proof}

The above can be rephrased as replacing constraints by extra variables in the integral. It also admits a generalization (see \cite{dworkdmod}, \cite{dworkdmod2}) from the case of a hypersurface to the case of higher codimension, though the proof is no longer as straightforward.  Namely,
let $X\subset Y$ be a pair of smooth varieties over the field $F$ of characteristic 0 where $F$ is for example $\mathbb{R}$, $\mathbb{C}$ or the $p$-adic field.  Furthermore, for the sake of concreteness assume that $Y=F^n$.  Let $X$ be cut out of $Y$ by the functions $f_1,..., f_m\in F[x^1,..., x^n]$ satisfying a suitable regularity condition as explained below.  Recall that an $F$-point of $X$ is an $n$-tuple $(r_1,..., r_n)$ of elements of $F$ satisfying $f_i(r_1,..., r_n)=0$ for all $i$. Let us require that for every $F$-point of $X$, the $m\times n$ matrix with $F$-entries $$M=\left(\partial_{x_j}f_i(r_1,..., r_n)\right)$$ has full rank.  More precisely, it is surjective as a map of $F$-vector spaces $$M:F^n\rightarrow F^m.$$

Let $f$ be an arbitrary function on $Y$, i.e. $f\in F[x^1,..., x^n]$, and let us consider its restriction to $X$.  Introduce a new function $g$ on $Y\times F^m$ by setting $$g=f+t^1 f_1+...+t^m f_m$$ where $t^i$ are the coordinates on $F^m$.

Then it follows for example from \cite{dworkdmod} that $$\mathcal{H}^i_f(X)\cong\mathcal{H}^{i+2m}_g(Y\times F^m)$$ for all $i$. The map from  $\mathcal{H}_f(X)$ to the shift of $\mathcal{H}_g(Y\times F^m)$ can be written down explicitly as $$\omega\mapsto\widetilde{\omega}\,df_1 dt_1 ... df_m dt_m$$ where $\widetilde{\omega}$ is the lifting  to $Y$ of the form $\omega$ on $X$. So that integrals over a non-linear $X$ can be replaced by integrals over the linear $Y\times F^m$.

\begin{remark}
The considerations of this section  cannot be applied to the
functions on a superspace. In this case one should work
with integral forms introduced in \cite{BL} instead of differential forms.
Another possibility is to fix a volume element and to work
with polyvector fields. (In the superspace case this data specifies
an integral form that can be integrated over a subspace of
codimension $k$ where $k$ is the number of indices of the
polyvector field.) It seems that this approach is also appropriate
in the infinite-dimensional case.
\end{remark}

\section {Frobenius map}
One can use the twisted de Rham cohomology to construct the Frobenius
map on the p-adic cohomology (cohomology with coefficients in
$K=\mathbb{C}_p$). By definition the  Frobenius map transforms a
point $x=(x_1, ..., x_n)\in K^n$ into a point $x^p=(x_1^p,...,x_n^p)$. If $f$ is a polynomial on $K^n$ the Frobenius map induces a map  $\psi$ sending
      $\mathcal{H}_f$ into $\mathcal{H}_{f'}$ where $f'(x)=f(x^p)$.
  However, we would like to modify the definition of the
  Frobenius map in such a way that it transforms the cohomology
  group into itself. This modification is based on a remark that
  one can find a p-adic number $\pi$ such that the
  expression $e^{\pi(z^p-z)}$ considered as a series with respect to $z$ is overconvergent (thus we are no longer speaking of $\mathcal{H}_f$, but rather of $\mathcal{H}_f^\dagger$) and what is
  equally important, we still have the result that the cohomology of the $D$-module $e^{\pi tP}$ (with \emph{overconvergent} forms) computes
  the cohomology of the hypersurface $P=0$ when it is smooth.
The appropriate $\pi$ is found as the solution of the
equation $\pi ^{p-1}=-p$ (see \cite{katz} for more details).

  We can define the Frobenius map $\Psi$ on the p-adic cohomology
  $\mathcal{H}_{\pi f}^\dagger$
as a map induced by the transformation of differential forms
 sending a form $\omega$ into a form $e^{\pi(f(x^p)-f(x))}\omega'$
where $\omega'$ is obtained from $\omega$ by means of the change of
variables $x\to x^p$. Here we use the fact that  $\mathcal{H}_{\pi
f}^\dagger$ allows differential forms with overconvergent
coefficients and the multiplication by  $e^{\pi(f(x^p)-f(x))}$
transforms a form of this kind into another form of the same kind.
One can say that the Frobenius map $\Psi$ is obtained from the
``naive" Frobenius map $\psi$ by introducing a ``correcting factor"
$e^{\pi(f(x^p)-f(x))}$.

Following the above, we can construct the Frobenius map on the p-adic cohomology
of a hypersurface $P(x_1,...,x_n)=0$ where $P$ is a polynomial
with p-adic coefficients. We identify this cohomology with the
twisted de Rham cohomology corresponding to the function
$f=t P(x_1,...,x_n)$ as before. It is a known fact (see \cite{katz} for example) that for this function the
cohomology  $\mathcal{H}_{\pi f}^\dagger $ is canonically isomorphic to  $\mathcal{H}_{ f} $. Thus the Frobenius map on $\mathcal{H}_{\pi f}^\dagger $ transfers to $\mathcal{H}_{ f} $ which is identified (up to shift) with  the p-adic cohomology of the hypersurface $P(x_1,...,x_n)=0$.

 This construction of the Frobenius map on the p-adic cohomology
 is equivalent to the original one provided by Dwork.
It is believed that Dwork's construction is equivalent to the more
modern construction via the crystalline cite that is used in \cite{inst} and explained in terms of
supergeometry in \cite{padicph}, but it seems that a complete proof of this
equivalence  does not exist in the literature.

Notice, that in the case when we have a family of hypersurfaces
labeled by a parameter $\lambda $, we can construct a Frobenius map in two different and inequivalent
ways.  Namely, we can either raise $\lambda$ to the $p$-th power or
not.  In the former case we have a Frobenius that acts preserving
the fibers of the family, and in the latter case we must modify the
correction factor to be $e^{\pi(f(x^p,\lambda^p)-f(x,\lambda))}$ and
now we get a Frobenius that mixes fibres.

\section {Topological Landau-Ginzburg model and topological sigma-models}

The main ingredient of a Landau-Ginzburg model is an algebraic family of
algebraic manifolds $X_{\lambda}$ equipped with a family
of algebraic functions $f_{\lambda}$. Here $\lambda$ runs over a manifold $\Lambda$ (the base of the family).  Denoting the union
of $X_{\lambda}$ by $X$ we obtain an algebraic function  $f$ on $X$ and a map $p$ from $X$ to $\Lambda$.  In the simplest case we
have a family of polynomials on the constant $X_\lambda=\mathbb{C}^n$.

Recall that for  every $\lambda \in \Lambda$ we can consider the twisted de Rham
cohomology $\mathcal{H}_{\lambda} = \mathcal{H}_{f_{\lambda}}$; under
some  conditions these cohomologies form a vector bundle over
$\Lambda$ that comes equipped with a flat connection (the Gauss-Manin connection).
For an arbitrary family, the collection of $\mathcal{H}_{\lambda}$'s is
naturally equipped with the structure of a D-module on the parameter
space $\Lambda$ of $\lambda$. The D-module structure can be viewed as a
flat connection  if this union forms the total space of a vector bundle
with fibers $\mathcal {H}_{\lambda}$. Another important ingredient
of a Landau-Ginzburg model  is a family of holomorphic volume elements
$\Omega_{\lambda}$ on the manifolds $X_{\lambda}$. A special case of a
Landau-Ginzburg model is a B-model; here the functions  $f_{\lambda}$ identically
vanish (in other words, a B-model is specified by a family of
Calabi-Yau manifolds).  More precisely, we are talking about a
genus zero Landau-Ginzburg model and a B-model. From the viewpoint of a
physicist one can define a B-model for an arbitrary genus as the quantization
of a genus zero B-model, however this definition is not
mathematically rigorous due to some ambiguities in the quantization
procedure.

Under certain conditions one can prove that the Landau-Ginzburg model specifies a Frobenius manifold (i.e., a genus 0 TQFT); in particular, for an appropriate family of volume elements this can be demonstrated for a miniversal deformation of a function having one isolated critical point. This is also true for a B-model on a family of compact manifolds.

The above definitions of a Landau-Ginzburg model and of a B-model
make sense over an arbitrary ring if the manifolds $X_{\lambda}$
are  defined over this ring with the one-forms $df_{\lambda}$ and the volume elements
$\Omega_{\lambda}$ having
coefficients belonging to the same ring as well.
In particular, we obtain in this way the definition of p-adic B-model used in \cite {inst}. The consideration of Sec 3 implies that the Frobenius map that was crucial in \cite {inst} can be defined in the case of Landau-Ginzburg model
over $\mathbb{C}_p$. (However, to prove integrality results one needs generalization of another definition of Frobenius map that is
applicable over $\mathbb{Z}_p$.) 

One can also consider an A-model on a manifold $Y$ defined over a field of characteristic zero. If the field is algebraically closed  the standard considerations permit us to relate the counting of algebraic curves to the homological calculations on the space of stable maps; otherwise we should modify the definitions by considering the intersections over the algebraic completion. Notice that instead of the K\"{a}hler metric one should  consider an element of the two-dimensional cohomology. Strictly speaking  this element should obey some conditions that guarantee that the expressions
we obtain are well defined at least as power series. We will disregard this subtlety here.

In the genus zero case, the A-model specifies the quantum multiplication on the cohomology; the structure coefficients $c^k_{ab}$ of this multiplication determine a family
$\nabla_a=\partial _a+zc^k_{ab}$ of flat connections on a trivial vector bundle over the two-dimensional cohomology of $Y$ (or, more generally over the total cohomology) with fiber the total cohomology of  $Y$.
If the genus zero A-model is defined on a complete intersection  $Y$
in a toric variety, then it is equivalent to a certain Landau-Ginzburg model  $X_{\lambda}, z^{-1}f_{\lambda},  \Omega _{\lambda}$. This statement (the mirror theorem) was proved by Givental
over the complex numbers, but in fact his proof works  over an arbitrary field of characteristic zero. It is also possible to formally derive the
statement of the mirror theorem over a field of characteristic zero from the mirror theorem over the complex numbers. More precisely,
one can construct a map from $\Lambda$ (the base of the family of
manifolds $ X_{\lambda}$) into the two-dimensional cohomology of $Y$ (the mirror map); this map can be lifted to an isomorphism of the vector bundles between the twisted de Rham cohomology groups $\mathcal {H}_{\lambda}$ over $\Lambda$ and the trivial vector bundle of cohomology groups of $Y$ over the two-dimensional cohomology of
$Y$. The Gauss-Manin connection on the Landau-Ginzburg side
corresponds to the connection $\nabla_a=\partial _a+zc^k_{ab}$ on the A-model side.
If a projective manifold is defined over the integers, then  the corresponding space of stable maps is also defined over the integers
and hence over an arbitrary ring. This means that we can define, at least formally,
an A-model for every ring, however there exists no clear relation between this definition and the counting of algebraic curves.

One can check that the correspondence between the A-model and the
Landau-Ginzburg model given by the mirror theorem remains valid for an arbitrary entire ring if  we neglect torsion, i.e., if we work with groups $\mathcal{H}'_f (K)$. Namely, using the universal coefficients  theorem we reduce the proof to the case of the ring of integers. In this case we should prove that
the correspondence between the relevant cohomologies preserves the integral structure in $\mathcal{H}_f (\mathbb{Q})$.
This follows from the integrality of the mirror map and from the remark that   the connections on both sides of the mirror correspondence are compatible with the integral structure.

{\bf Acknowledgments} We are indebted to
M. Kontsevich, M. Movshev, A. Ogus and V. Vologodsky for useful discussions.

\end{document}